\documentclass[11pt]{amsproc}
\usepackage{mathrsfs}
\usepackage{stmaryrd}
\usepackage{cases}
\usepackage{amssymb}
\usepackage{amsmath}
\usepackage{amsfonts}
\usepackage{graphicx}
\usepackage{amsmath,amstext,amsbsy,amssymb}
\usepackage[table]{xcolor}
\usepackage{multirow}
\usepackage{cases}

\newtheorem{theorem}{Theorem}[section]
\newtheorem{lemma}[theorem]{Lemma}

\newtheorem{proposition}[theorem]{Proposition}

\theoremstyle{definition}

\theoremstyle{remark}
\newtheorem{remark}[theorem]{Remark}

\numberwithin{equation}{section} \errorcontextlines=0

\newcommand{\la}{\lambda}

\begin{document}

\title[Irreducible projective characters of wreath products] 
{Irreducible projective characters of wreath products} 
\author{Xiaoli Hu}
\address{School of Sciences,
South China University of Technology, Guangzhou 510640, China}
\email{xiaolihumath@163.com}
\author{Naihuan Jing*}
\address{Department of Mathematics,
   North Carolina State University,
   Raleigh, NC 27695, USA}
\email{jing@math.ncsu.edu}
\keywords{wreath products, spin groups, projective characters}
\subjclass[2010]{Primary: 20C25; Secondary: 20C30, 20E22}
\thanks{*Corresponding author.}

\begin{abstract}
The irreducible character values of the spin wreath products $\widetilde{\Gamma}_n=\Gamma\wr \widetilde{S}_n$
of the symmetric group and a finite group $\Gamma$ are completely determined for arbitrary $\Gamma$.

\end{abstract}
\maketitle
\section{Introduction}

The symmetric group has been one of the core materials
in mathematics and theoretical physics since Frobenius developed
representation theory of finite groups.
Using the duality between $S_n$ and $\mbox{GL}_m$ Schur found the complete set of invariants for the general linear group and
revolutionized the theory of symmetric functions using the Schur functions in his dissertation.
In the thirties Specht \cite{Sp} generalized Schur and Frobenius'
 theory to the wreath products of the symmetric group and any finite
 group. More recently Macdonald reformulated Specht's theory in the classic monograph \cite{M}.

The double covering groups $\widetilde{S}_n$ of the symmetric group have many similar and interesting
properties as shown by Schur in \cite{S}. In that seminal paper Schur
generalized Frobenius's theory and introduced the famous Schur Q-functions.
Schur's character theory of $\widetilde{S}_n$ consists of two parts.
The first part of the character values on conjugacy classes associated to partitions
with odd integer parts
are exactly given by an analogous Frobenius formula in terms of the
Schur Q-functions; the second part of character values on strict partitions
was solved with the help of twisted tensor products of
Clifford algebras.
In the same direction,
 Morris \cite{Mo} formulated an iterative rule for computing
 the spin character values of the symmetric group,
 and Nazarov \cite{Na} constructed all irreducible representations of the spin group.
 J$\acute{o}$zefiak \cite{Jo} also computed the projective character values for
 a related double covering group of the hyperoctahedral group using similar techniques.

In \cite{FJW} the second author and collaborators
determined all irreducible character values on {\it even}
 conjugacy classes of the spin wreath products $\widetilde{\Gamma}_n$ by vertex operator calculus in the context of
 the spin McKay correspondence. The character values
 generalize the first
part of Schur's theory on $\widetilde{S}_n$. 
The key was to
show that the character values at conjugacy classes of even colored partitions (with odd integer parts) are
given by matrix coefficients of certain products of twisted vertex operators. However, the other
character values on
{\it odd} strict partition-valued functions could not be explained in the context of the McKay correspondence. To the authors'
knowledge, other available methods such as the Hopf algebraic approach \cite{HH} 
seem not helpful either.

The knowledge of the remaining character values for $\widetilde{\Gamma}_n$ would be an analog of the second part
in Schur's pioneering work \cite{S}.
In the case of $\Gamma$ being abelian, we solved the
problem in \cite{HJ} by using the Mackey-Wigner method of little groups (cf. \cite{Se})
to construct all irreducible spin representations
of the wreath products
(see also \cite{MJ} for some cases). However the method of little groups
does not work in the most general case
for arbitrary finite group $\Gamma$. It seems that a new method is
needed for determination of the irreducible characters.

The purpose of this paper is to complete the character theory of the spin wreath products $\widetilde{\Gamma}_n$ and
compute the missing part of the characters table of $\widetilde{\Gamma}_n$
for any finite group $\Gamma$. We will construct all irreducible characters
and in particular provide explicit formulas for
character values on the conjugacy classes of the second type.
It turns out that spin character values in this part can be
non-zero for more than two conjugacy classes in contrast of Schur's case, nevertheless they are still sparsely zero.
We note that the exhaustion method is used to pick up all non-zero character values, which
bears some similarity to Schur's original method.

\section{Projective representations of $\widetilde{\Gamma}_n$}
\subsection{Spin wreath products} According to Schur, there are two non-isomorphic double covering groups
of the symmetric group $S_n$ when $n\geq 4$ and $n\neq 6$. But their representations are in complete one-to-one correspondence.
We fix one of them, and let the spin symmetric  group $\widetilde{S}_n$ be the finite group generated by
$z$ and $t_i, ~(i=1, \cdots, n-1)$ with the relations:
\begin{equation}
\begin{split}
 z^2=1, ~t_i^2=(t_it_{i+1})^3=z, ~
zt_i=t_iz, ~ t_it_j=zt_jt_i, \quad (|i-j|>1).\\
\end{split}
\end{equation}

Let  $\theta_n$ be the homomorphism from $\widetilde{S}_n$
to $S_n$ sending $t_i$ to the transposition $(i,i+1)$ and $z$ to 1. This says that $\widetilde{S}_n$
is a central extension of $S_n$ by the cyclic group $\mathbb{Z}_2$.

The spin group $\widetilde{S}_n$ has a cycle presentation $\grave{a}$ la Conway \cite{Ws}.
For $i<j$, the transposition $[ij]$ is deinfed as $[ij]=t_{j-1}\cdots t_{i+1}t_it_{i+1}\cdots t_{j-1}$.
Here if $j=i+1$, we take $t_i=[i,i+1]$.
Then for $i_1<i_2<\cdots<i_k\leq n$, we define
$[i_1i_2\cdots i_k]=[i_1i_2][i_2i_3]\cdots[i_{k-1}i_k]$.
Finally for a permutation $\sigma\in S_k$ we define
\begin{equation}
[i_{\sigma(1)}i_{\sigma(2)}\cdots i_{\sigma(k)}]=z^{l(\sigma)}[i_1i_2\cdots i_k].
\end{equation}

Given a permutation $w\in S_n$. We can fix its cycle product as follows. First each cycle is written as a word lexicographically by rotating its content, then we rearrange the order of the cycles lexicographically to obtain
a unique presentation $w=\prod_{i=1}^{l}(a_{i1}\cdots a_{i\la_i})$, where $\lambda_i$ are the lengths of the cycles.
We then define the element
$t_{w}=\prod_{i=1}^{l}[a_{i1}\cdots a_{i\la_i}]\in\widetilde{S}_n$.
Note that $\theta_n^{-1}(w)=z^{p}t_{w}$. One also has that
$t_{w_1}t_{w}t_{w_1}^{-1}=z^{p}t_{w_1ww_1^{-1}}$ for $w,w_1\in S_n$, where $p=0$ or $1$.
Similarly for a partition $\rho$ we denote $t_{\rho}=t_{w(\rho)}$, where $w(\rho)=(1\cdots \rho_1)\cdots (n-\rho_l+1\cdots n)$.

For a positive integer $n$ and a finite group $\Gamma$, let $\Gamma^n=\Gamma\times\cdots\times
\Gamma$ be the $n$-fold direct product of $\Gamma$, and let $\Gamma^0=1$. The spin group
$\widetilde{S}_n$ acts on $\Gamma^n$ by permuting the components:
\begin{equation}
\begin{split}
t_{w}(g_1,\cdots,g_n)=(g_{w^{-1}(1)},\cdots,g_{w^{-1}(n)}),~
z(g_1,\cdots,g_n)=(g_1,\cdots,g_n).
\end{split}
\end{equation}
The spin wreath product $\widetilde{\Gamma}_n=\Gamma\wr
\widetilde{S}_n$ is the semi-direct product
$$\widetilde{\Gamma}_n=\Gamma^n\rtimes \widetilde{S}_n=\{(g,t)|g=(g_1,\cdots,g_n)\in\Gamma^n, t\in\widetilde{S}_n\}$$
with the multiplication
$(g,t)\cdot(h,s)=(gt(h),ts).$
The quotient group $\widetilde{\Gamma}_n/\langle z\rangle$ is isomorphic to the semidirect product
$\Gamma_n=\Gamma^n\rtimes S_n$, and the canonical homomorphism  $\theta_n$  from $\widetilde{\Gamma}_n$ to $\Gamma_n$ sends $(g,t_i)\mapsto (g,(i,i+1))$ and $(g,z)\mapsto (g,1)$. For simplicity we have used the same symbol
for the homomorphism $\theta_n$ for the wreath product.

We can define a parity $p$ for $\widetilde{\Gamma}_n$ by (here $g\in\Gamma^n, t_i\in \widetilde{S}_n$)
$$p(g,t_i)=1\, (1\leq i\leq n-1), \ \ \ \ p(g,z)=0.$$
This agrees with the usual parity for $\widetilde{S}_n$ when $\Gamma$ is the trivial group. Therefore the group algebra $\mathbb{C}[\widetilde{\Gamma}_n]$ has a superalgebra structure, and $\mathbb{C}[\Gamma\wr \widetilde{A}_n]$ is the even subspace.

\subsection{Conjugacy classes}
Let $\Gamma_{*}=\{c^i|i=0,1,\dots,r\}$ be the set of conjugacy classes
of $\Gamma$ and denote by $\Gamma^*=\{\gamma_i|i=0,1,\cdots,r\}$ the set of
irreducible characters of $\Gamma$. Let $\zeta_c$  be the order of
the centralizer of an element in the conjugacy class $c\in \Gamma_*$, then the
order of the conjugacy class $c$ is $|\Gamma|/\zeta_c$. Here for a finite set $X$ we denote by $|X|$ its cardinality. In the following we
follow Macdonald's notations \cite{M}.

A partition-valued function $\rho=(\rho(c))_{c\in\Gamma_{*}}$ defined on $\Gamma_{*}$
consists of $|\Gamma_*|$ partitions indexed by conjugacy classes $c\in\Gamma_*$. The
weight of $\rho$ is defined by $||\rho||=\sum_{c\in\Gamma_{*}}|\rho(c)|$,
and the length is given by $l(\rho)=\sum_{c\in\Gamma_{*}}l(\rho(c))$. It helps to
visualize $\rho$ as a colored partition in which each sub-partition
$\rho(c)$ is colored by $c$. Let
 $ \mathcal {P}(\Gamma_{*})$ be the set of partition-valued functions indexed by $\Gamma_{*}$.
It is well-known that the conjugacy classes of $\Gamma_n$ are parameterized
by $\mathcal {P}(\Gamma_{*})$. For
an element $(g, \sigma)\in \Gamma_n$, the permutation $\sigma$ gives rise to a
cycle partition $\la=(\la_1, \la_2, \ldots)$. For each part $\la_i=k$, which corresponds to
the cycle
$(i_1i_2\cdots i_k)$, we associate the cycle-product $g_{i_k}g_{i_{k-1}}\cdots
g_{i_1}\in \Gamma$. If the cycle-product belongs to the conjugacy class
$c$, then we color this part $\la_i$ by $c$ which turns
$\la$ into a colored partition. In this way we get the
parametrization of conjugacy classes of $\Gamma_n$ by $ \mathcal {P}(\Gamma_{*})$.
For
$\rho=(\rho(c))_{c\in\Gamma_*}\in \mathcal {P}_n(\Gamma_*)$, let $C_{\rho}$ be the corresponding conjugacy class in $\Gamma_n$.

We denote by $\mathcal {SP}(\Gamma_{*})$ the set of
partition-valued functions $(\rho(c))_{c\in \Gamma_{*}}$ in $ \mathcal
{P}(\Gamma_{*})$ such that each partition $\rho(c)$ is {\it strict}, i.e. $\rho(c)$
has distinct parts. Let $\mathcal
{OP}(\Gamma_{*})$ be the set of partition-valued functions $(\rho(c))_{c\in
\Gamma_{*}}$ on $\Gamma_{*}$ such that all parts of the partitions
$\rho(c)$ are odd integers.

For each partition $\la$ we define the {\em parity} $d(\la)=|\la|-l(\la)$.
Similarly, for a partition-valued function $\rho=(\rho(c))_{c\in
\Gamma_{*}}$, we define $d(\rho)= \|\rho\|-l(\rho)$. Then $\rho$ is {\it even} (resp. {\it odd}) if
$d(\rho)$ is even (resp. odd). We set ${\mathcal{P}}_n^0(\Gamma_{*})$ (resp. ${\mathcal{P}}_n^1(\Gamma_{*})$) to be the collections
of even (resp. odd) partition-valued functions  $\rho$ on $\Gamma_{*}$ such that $||\rho||=n$.
As a convention we denote $\mathcal {SP}_n^i(\Gamma_{*})=\mathcal
{P}_n^i(\Gamma_{*})\cap \mathcal {SP}(\Gamma_{*})$ and $\mathcal {OP}_n(\Gamma_{*})=\mathcal
{P}_n(\Gamma_{*})\cap \mathcal {OP}(\Gamma_{*})$ for $i\in \{0,1\}$.  When $\Gamma_{*}$ just
consists of a single element,
$\mathcal{P}(\Gamma_{*})$ will be simply written as $\mathcal {P}$.
Similarly we have notations
such as $\mathcal {OP}_n$, $\mathcal{SP}_n$, and $\mathcal
{SP}_n^i$.

\subsection{Split conjugacy classes}
 An element $\widetilde{x}\in\widetilde{\Gamma}_n$ is called
{\em non-split} if $\widetilde{x}$ is conjugate to $z\widetilde{x}$.
Otherwise $\widetilde{x}$ is said to be {\em split.}
 An element $x\in\Gamma_n$ is called split if
$\theta_n^{-1}(x)$ is split. A conjugacy class of $\widetilde{\Gamma}_n$ is called split if its
elements are split. It is known that the conjugacy class $C_{\rho}$ of
$\Gamma_n$ splits if and only if the preimage
$\theta_n^{-1}(C_{\rho})=:D_{\rho}$ splits into two
conjugacy classes in $\widetilde{\Gamma}_n.$

For a partition $\la=(1^{m_1}2^{m_2}3^{m_3}\cdots)$ of $n$, we denote by
$z_{\la}=\prod_{i\geq 1}i^{m_i}m_i!$ the order of the centralizer of
the permutation with cycle type $\la$ in $S_{n}.$
For each partition-valued function
$\rho=(\rho(c))_{c\in\Gamma_*}$, we define
$$Z_{\rho}=\prod_{c\in\Gamma_*}z_{\rho(c)}\zeta_c^{l(\rho(c))},
$$
which is the order of the centralizer of an element of
conjugacy type $\rho=(\rho(c))_{c\in\Gamma_*}$ in $\Gamma_n$. 
The order of the centralizer of an element of conjugacy type $\rho$ in $\widetilde{\Gamma}_n$ is given by
\begin{equation}\nonumber
\widetilde{Z}_{\rho}=\left\{ \begin{aligned}
         2Z_{\rho}, & \ \ \ C_{\rho} ~\hbox{is split},  \\
          Z_{\rho},& \ \ \ C_{\rho} ~\hbox{is non-split}.
                          \end{aligned} \right.
\end{equation}
For each split conjugacy class $C_{\rho}$ in $\Gamma_n$, we define
the conjugacy class $D_{\rho}^{+}$ in $\widetilde{\Gamma}_n$ to be
the conjugacy class containing the element $(g,t_{\rho})$ and define
$D_{\rho}^{-}=zD_{\rho}^{+}$, then
$ D_{\rho}=D_{\rho}^{+}\cup
D_{\rho}^{-}$.

As usual a representation $\pi$ of $\widetilde{\Gamma}_n$ is called
{\em spin} if $\pi(z)=-id$, then its character is a projective character of
$\Gamma_n$. It is clear that the character of
a spin representation of $\widetilde{\Gamma}_n$ are determined by
its values on split conjugacy classes. By a
standard result \cite{R, FJW} the conjugacy class of
$\widetilde{\Gamma}_n$ is split in $\widetilde{\Gamma}_n$ if and
only if either $\rho\in \mathcal {OP}_n(\Gamma_*)$ or $\rho\in
\mathcal {SP}_n^1(\Gamma_*)$.
In particular the split conjugacy classes of
$\widetilde{S}_n$ (i.e. when $\Gamma$ is the trivial group) are parameterized either by partitions with odd
integers or by odd strict partitions. By Euler's theorem the
number of strict partitions is equal to the number of odd partitions, therefore
the total number of such conjugacy classes of $\widetilde{\Gamma}_n$ are given by
$
|\mathcal{SP}_n^0(\Gamma_*)|+2|\mathcal{SP}_n^1(\Gamma_*)|.$

\section{The irreducible spin character table of $\widetilde{\Gamma}_n$}
\subsection{Schur's theory of $\widetilde{S}_n$}
Like the symmetric group $S_n$, nontrivial projective (spin) characters of $\widetilde{S}_n$ are parameterized by strict partitions $\la$ of $n$. One can classify
spin characters into the so-called {\it double spin} and
{\it associate spin} characters. The associated character $\chi'$ of a spin character
is defined to be $\chi'=sgn\cdot\chi$, where $sgn$ is the sign character.
If $\chi'=\chi$, then we say $\chi$ is a double spin
character (or self-associated).
For $\nu\in \mathcal{SP}_n$ and $d(\nu)=n-l(\nu)$ even, there
corresponds a unique irreducible (double) spin character $\Delta_{\nu}$;
for $d(\nu)$ odd, there corresponds a pair of irreducible (associate) spin characters
$\Delta_{\nu}^{\pm}$ for $\widetilde{S}_n$. Schur \cite{S}
showed that there was an analogous Frobenius formula
for the spin character values at the even conjugacy classes indexed by
partitions with odd integers.
For $\nu\in \mathcal{SP}$ the Schur $Q$-function $Q_{\nu}$ is defined by
$$Q_{\nu}(x_1, \cdots, x_n)
=2^{l}\sum_{w\in S_n/S_{n-l}}x_{w(1)}^{\nu_1}\cdots
x_{w(n)}^{\nu_n} \prod_{i<j}\frac{x_{w(i)}+x_{w(j)}}{x_{w(i)}-x_{w(j)}},
$$
where $n\geq l=l(\nu)$ and $S_{n-l}$ acts on $x_{l+1},\cdots, x_n$. It is known that
$Q_{\nu}$ is a polynomial in power sum symmetric functions $p_1, p_3, p_5, \cdots.$
As usual we will denote $p_{\alpha}=p_{\alpha_1}p_{\alpha_2}\cdots$ for a partition $\alpha$.

Schur showed that nontrivial values of the spin character $\Delta_{\nu}$ at even conjugacy classes 
are given by
\begin{equation}\label{Q}
Q_{\nu}=\sum_{\alpha\in\mathcal {OP}_n}2^{[\frac{l({\nu})+l(\alpha)+\bar{d}(\nu)}{2}]}
z_{\alpha}^{-1}\Delta_{\nu}(\alpha)p_{\alpha},
\end{equation}
where $[a]$
denotes the largest integer $\leq a$ and $\bar{d}(\nu)$ is equal to 0 (resp. 1) if $d(\nu)$ is even  (resp. odd).
Schur further proved the following results.
\begin{theorem} \cite{S} \label{Schur}
(i) For $n-l(\nu)$ even, the character $\Delta_{\nu}$ of $\widetilde{S}_n$ is determined by $\{\Delta_{\nu}(\alpha)|\alpha\in \mathcal{OP}_n\}$ (given by Eq.(\ref{Q})) and $\Delta_{\nu}(\mu)=0$ for $\mu\notin\mathcal{OP}_n$.

(ii) For $n-l(\nu)$ odd, the character $\Delta_{\nu}$ of $\widetilde{S}_n$ is determined by $\{\Delta_{\nu}(\alpha)|\alpha\in \mathcal{OP}_n\}$ (given by Eq.(\ref{Q})) and $\Delta_{\nu}(\nu)=(\sqrt{-1})^{(n-l(\nu)+1)/2}\sqrt{\nu_1\cdots\nu_k/2}$ for $\nu=(\nu_1,\cdots,\nu_k)$;
$\Delta_{\nu}(\mu)=0$ for $\mu\neq\nu$ and $\mu\in \mathcal{SP}_n^1$. Moreover, $(\Delta_{\nu})^{'}(\alpha)=\Delta_{\nu}(\alpha)$ for $n-l(\alpha)$ even and $(\Delta_{\nu})^{'}(\alpha)=-\Delta_{\nu} (\alpha)$ for $n-l(\alpha)$ odd.
\end{theorem}

For $\widetilde{\Gamma}_n$, the values of nontrivial spin characters at the conjugacy classes associated to $\rho$ for $\rho\in\mathcal{OP}_n(\Gamma_{*})$ are given by Frenkel-Jing-Wang \cite{FJW}.
\begin{lemma}\label{lema} For an irreducible spin $\widetilde{\Gamma}_n$-character $\chi_{\la}$ with type $\la\in\mathcal{SP}_n^1(\Gamma^{*})$
 and a subset $G$ 
 of $(\widetilde{\Gamma}_n)_*$, set $\langle  \chi,\chi\rangle_{G}=\sum_{D_{\rho}\in G}\frac{1}{|\widetilde{Z}_{\rho}|}|\chi(D_{\rho})|^2$, then $\langle  \chi,\chi\rangle_{\mathcal{OP}_n(\Gamma_{*})}=\langle  \chi,\chi\rangle_{\mathcal{SP}^1_n(\Gamma_{*})}=\frac{1}{2}$ (see \cite{FJW}).
 \end{lemma}
 \begin{proof}
 Since $\chi_{\la}^{'}(x)=(-1)^{deg(x)}\chi_{\la}(x)$ for $x\in\widetilde{\Gamma}_n$, we have that
 \begin{equation}
 \begin{split}
 2=&\langle\chi_{\la}+\chi_{\la}^{'},\chi_{\la}+\chi_{\la}^{'}\rangle_{\widetilde{\Gamma}_n}\\
 =&(\sum_{\rho\in \mathcal{OP}_n(\Gamma_{*})}+\sum_{\rho\in \mathcal{SP}^1_n(\Gamma_{*})})\frac{1}{\widetilde{Z}_{\rho}}|(\chi_{\la}+\chi_{\la}^{'})(D_{\rho})|^2\\
 =&\sum_{\rho\in \mathcal{OP}_n(\Gamma_{*})}\frac{1}{\widetilde{Z}_{\rho}}\cdot4|\chi_{\la}(D_{\rho})|^2\\
  \end{split}
 \end{equation}
 Therefore, $\sum_{\rho\in \mathcal{OP}_n(\Gamma_{*})}\frac{1}{\widetilde{Z}_{\rho}}|\chi_{\la}(D_{\rho})|^2=
 \sum_{\rho\in \mathcal{SP}_n^1(\Gamma_{*})}\frac{1}{\widetilde{Z}_{\rho}}|\chi_{\la}(D_{\rho})|^2=\frac{1}{2}.$
 \end{proof}
Table 1 shows
the status of character values. Part $D$ indicates what we will compute in
this paper: $\chi(D^{\pm}_{\rho})$ for $\rho\in \mathcal{SP}_n^1(\Gamma_{*})$.
\begin{table}[!h]
\tabcolsep 0pt
\caption{The spin character table for $\widetilde{\Gamma}_n$}
\vspace*{-15pt}
\begin{center}
\def\temptablewidth{0.95\textwidth}
{\rule{\temptablewidth}{0.8pt}}
\begin{tabular*}{\temptablewidth}{@{\extracolsep{\fill}}ccc}
 Character \ $\backslash$ Class : &  $\rho\in\mathcal{OP}_n({\Gamma_{*}})$ & $\rho\in\mathcal{SP}^1_n(\Gamma_{*})$  \\   \hline
 $\la\in\mathcal{SP}^0_n(\Gamma_{*}), \chi_{\la}$:  & {A:}~ known & { C: 0}\\
   $\langle  \chi_{\la},\chi_{\la}\rangle$:  & 1 & 0\\ \hline
  $\la\in \mathcal{SP}^1_n(\Gamma_{*}), \chi_{\la}$:  &{B:}~ known & {D:}~this paper \\
  $\langle  \chi_{\la},\chi_{\la}\rangle$:  & 1/2 & 1/2 \\
       \end{tabular*}
       {\rule{\temptablewidth}{0.8pt}}
       \end{center}
       \end{table}

\subsection{Decomposition of colored partitions}
When $\Gamma$ is a finite abelian group, we used the Mackey-Wigner method of little groups
to decompose the action of $S_n$ on the characters of $\Gamma^n$. It turns out that
the invariant subgroup of each $S_n$-orbit is a Young subgroup of $S_n$ and vice versa.
Then we can construct all spin irreducible representations
indexed by strict partition-valued functions by induction, and show that the character values
are sparsely zero and the non-zero values are given according to how the partitions are supported on various conjugacy classes (see \cite{HJ}).
This method is no longer available when $\Gamma$ is an arbitrary finite group.
Next we use a different method to compute
spin character values on odd strict partition-valued functions for a general finite group $\Gamma$.

We discuss the conjugacy classes generated by Young subgroups. For $\nu=(\nu_{\gamma})_{\gamma\in\Gamma^{*}}\in\mathcal{SP}(\Gamma^*)$, we first erase the coloring of $\nu_{\gamma_i}$ and reassign
colors arbitrarily from $\Gamma_*$. Suppose $\nu_{\gamma_i}=(\nu^i_1, \nu^i_2, \cdots, \nu^i_s)$
as an ordinary strict partition. Then for any composition $c(\underline{i})=(c^{i_1}, \cdots, c^{i_s})$ from $\Gamma_*=\{ c^0, \cdots, c^r\}$, we
assign the colors $c^{i_1}, \cdots, c^{i_s}$ consecutively to the parts of $\nu_{\gamma_i}$ to get a colored
partition $\nu^i(\underline{c})$ of the shape $\nu^i$. The resulted multi-colored partition
will be denoted by $\nu^{I}$, where $I\in\{0, \cdots, r\}^{l(\nu)}$. We let $[\nu]$ be the
collection of all these multi-colored partitions, which are parameterized by mappings $I: \{1, 2, \cdots, l(\nu)\}\to \{0, \cdots, r\}$. Therefore the cardinality of $[\nu]$ is $(r+1)^{l(\nu)}$. See the following example.
\begin{equation*}
\left(
\begin{aligned}
{\color{black}\begin{tabular}{|c|c|c|}\hline &~    & ~  \\
\hline &  \\\cline{1-2} ~
\\ \cline{1-1}
\end{tabular}},\quad
&{\color{white}\begin{tabular}{|c|c|c|}\hline\rowcolor{black}~~~&~~~ &~~~ \\
\hline \rowcolor{black}& \\\cline{1-2}
\end{tabular}}\\
\nu(\gamma_0)\quad\quad  & ~\nu(\gamma_1)
\end{aligned}
\right)
\rightsquigarrow
\left(\begin{aligned}
\begin{aligned}
&{\color{black}\begin{tabular}{|c|c|c|}\hline\rowcolor{white} ~&~    & ~  \\\cline{1-3}\end{tabular}} \\
&{\color{white}\begin{tabular}{|c|c|}\hline \rowcolor{black}  & ~  \\\cline{1-2}\end{tabular}} \\
&{\color{black}\begin{tabular}{|c|}\hline\rowcolor{white} \\\cline{1-1}\end{tabular}} \end{aligned},\quad
&\begin{aligned}
&{\color{white}\begin{tabular}{|c|c|c|}\hline \rowcolor{black}& & \\\cline{1-3}\end{tabular}} \\
&{\color{black}\begin{tabular}{|c|c|}\hline \rowcolor{white}  & \\\cline{1-2}\end{tabular}} \\
\end{aligned}\\
\nu^1({\underline{c}})\quad \quad& \nu^2({\underline{c}^{'}})\\
\end{aligned}
\right),~
 \left(  \begin{aligned}
\begin{aligned}
&{\color{white}\begin{tabular}{|c|c|c|}\hline \rowcolor{black}&~    & ~  \\\cline{1-3}\end{tabular}} \\
&{\color{black}\begin{tabular}{|c|c|}\hline \rowcolor{white}  & ~  \\\cline{1-2}\end{tabular}} \\
&{\color{black}\begin{tabular}{|c|}\hline \rowcolor{white} \\\cline{1-1}\end{tabular}} \end{aligned},\quad
&\begin{aligned}
&{\color{black}\begin{tabular}{|c|c|c|}\hline\rowcolor{white} &~ & \\\cline{1-3}\end{tabular}} \\
&{\color{white}\begin{tabular}{|c|c|}\hline\rowcolor{black}   &  \\\cline{1-2}\end{tabular}} \\
\end{aligned}\\
\nu^1({\underline{c}})\quad \quad & ~\nu^2({\underline{c}^{'}})\\
\end{aligned}    \right),\cdots
\end{equation*}

\subsection{Spin supermodules vs. spin modules}
A spin $\widetilde{\Gamma}_n$-module $V$ becomes a spin supermodule when $ch_V(x)=0$ for all odd elements $x$. According to
\cite{Jo} there are two basic types of simple supermodules: type $M$ or $Q$, corresponding to
our double spin and a pair of associated spin modules when forgetting the $\mathbb Z_2$-gradation.
Moreover all double spin and associate spin modules are realized in this way.

For  a strict partition-valued function  $\la=(\la_{\gamma})_{\gamma\in\Gamma^{*}}\in \mathcal{SP}_n(\Gamma^{*})$, let
$J_{\la}=\{\gamma\in \Gamma^{*}|\la_{\gamma}\mbox{~is  odd strict}\}$ and
$J_{\la}^{'}=\Gamma^*-J_{\lambda}=\{\gamma\in\Gamma^{*}|\la_{\gamma}\mbox{~is even strict}\}$.
Let $U_{\gamma}$ be the irreducible $\Gamma$-module associated with the irreducible character $\gamma\in\Gamma^*$. Suppose $V_{\la_{\gamma}}$ is the  irreducible
spin $\widetilde{S}_{|\la_{\gamma}|}$-supermodule determined by the strict partition $\la_{\gamma}$, then $U_{\gamma}^{\otimes |\la_{\gamma}|}\otimes V_{\la_{\gamma}}$
is a spin $\widetilde{\Gamma}_{|\la_{\gamma}|}$-supermodule.
Let $\widetilde{\Gamma}_{\la}$ be the non-trivial double cover of the Young subgroup $\Gamma^{n}\rtimes S_{(|\la_{\gamma_0}|, \cdots, |\la_{\gamma_r}|)}$. Then by \cite{FJW}, the super tensor product
$$\hat{\bigotimes}_{\gamma\in\Gamma^{*}}(U_{\gamma}^{\otimes |\la_{\gamma}|}\otimes V_{\la_{\gamma}})$$
decomposes completely into $2^{[\frac{|J_{\la}|}{2}]}$ copies of an
irreducible spin $\widetilde{\Gamma}_{\la}$-super\-module. Denote this
irreducible supermodule by $W_{\la}$.

\begin{proposition} \label{super-ord}
The underlying $\widetilde{\Gamma}_n$-module of the induced supermodule
$\mbox{Ind}_{\widetilde{\Gamma}_{\la}}^{\widetilde{\Gamma}_n}W_{\la}$ is
an irreducible double spin $\widetilde{\Gamma}_n$-module or a direct sum
of two associated irreducible ${\widetilde{\Gamma}_n}$-modules according to the supermodule is of type $M$ or $Q$ respectively.
In terms of partitions this corresponds to whether $d(\la)=n-l(\lambda)$ is even or odd.
\end{proposition}
\begin{proof} This is true in a more general context. Each irreducible $\widetilde{\Gamma}_n$-supermodule of
type $M$ (resp. $Q$)
is an irreducible double spin (resp. a pair of associated spin) ${\widetilde{\Gamma}_n}$-module(s).
Suppose that the underlying $\widetilde{\Gamma}_n$-module of
our irreducible $\widetilde{\Gamma}_n$-supermodule $V=\mbox{Ind}_{\widetilde{\Gamma}_{\la}}^{\widetilde{\Gamma}_n} W_{\lambda}$ decomposes into a direct sum of
irreducible  $\widetilde{\Gamma}_n$-modules:
$$V=\sum_{i=1}^mV_i\oplus \sum_{j=1}^q(W_j\oplus W_j'),
$$
where $V_i$ are irreducible double spin modules, and $W_j$ and $W_j'$
are irreducible associate spin modules.
It follows from general theory \cite{J} of double spin and associate spin modules
 that, as an $\Gamma\wr\widetilde{A}_n$-module, $Res(V_i)$ decomposes into $V_i'\oplus V_i''$, while
$Res(W_j)$ or $Res(W_j')$ remains irreducible. Thus we will have
\begin{equation*}
\langle V, V\rangle_{\Gamma\wr\widetilde{A}_n}=2m+2q.
\end{equation*}
On the other hand we know that $\langle V, V\rangle_{\widetilde{\Gamma}_n}=1$ or $2$ according to
the spin supermodule $V$ being of type $M$ or $Q$ by vertex operator calculus
\cite{FJW}. This means that $\langle V, V\rangle_{\Gamma\wr\widetilde{A}_n}=2$, so we must have that
either $m=1$ or $q=1$.
\end{proof}

\subsection{Construction of irreducible spin characters}
Let $\tilde{\Delta}_{\la_{\gamma}}$ be the character of the supermodule $V_{\la_{\gamma}}$, then $\hat{\bigotimes}_{\gamma\in\Gamma^{*}}(\gamma^{\otimes |\la_{\gamma}|}\otimes \tilde{\Delta}_{\la_{\gamma}}))$ is the character of super tensor product $\hat{\bigotimes}_{\gamma\in\Gamma^{*}}(U_{\gamma}^{\otimes |\la_{\gamma}|}\otimes V_{\la_{\gamma}})$.
In order to  consider  the ordinary spin characters of $\widetilde{\Gamma}_{\la}$, Schur defined  the starred tensor product
$\circledast_{\gamma\in\Gamma^{*}}({\gamma}^{\otimes |\la_{\gamma}|}\otimes \Delta_{\la_{\gamma}})$ to be  an underlying ordinary irreducible component of $\hat{\bigotimes}_{\gamma\in\Gamma^{*}}({\gamma}^{\otimes |\la_{\gamma}|}\otimes \tilde{\Delta}_{\la_{\gamma}})$ (see \cite{S}).
 Let $\tilde{\chi}_{\la}$ (resp. $\tilde{\Delta}_{\la_{\gamma}}$) be the  character of irreducible supermodule $W_{\la}$ (resp. $V_{\la_{\gamma}}$). It is known that $\tilde{\chi}_{\la}$ (resp. $\tilde{\Delta}_{\la_{\gamma}}$ ) is also an irreducible ordinary character when $d(\la)$ (resp. $d(\la_{\gamma})$) is even. We denote it by $\chi_{\la}$ (resp. $\Delta_{\la}$) when it is regarded as an ordinary character. While   $\tilde{\chi}_{\la}$ (resp. $\tilde{\Delta}_{\la_{\gamma}}$ )  decomposes into two ordinary irreducible characters $\chi_{\la}$ and $\chi_{\la}^{'}$ (resp. $\Delta_{\la_{\gamma}}$ and $\Delta^{'}_{\la_{\gamma}}$)  when  $d(\la)$ (resp.  $d(\la_{\gamma})$) is odd.

\begin{remark}\label{lem}
 Let $\rho=\rho^0\cup\cdots\cup\rho^r$ be a union of partition-valued functions on $\Gamma_{*}$ whose parts are union of those of
partition-valued functions $\rho^0,\cdots,\rho^r$ on $\Gamma_{*}$. Then $D_{\rho}$ splits into two
conjugacy classes of $\widetilde{\Gamma}_{(|\rho^0|,\cdots,|\rho^r|)}$ if and only if
\\
(1) $\rho=\rho^0\cup\cdots\cup\rho^r\in \mathcal{OP}_{|\rho^0|+\cdots+|\rho^r|}(\Gamma_{*})$ or
\\
(2) $\rho^i\in\mathcal{SP}_{|\rho^i|}(\Gamma_{*})$ for $0\leq i\leq r$ and $||\rho||-l(\rho)$ is odd.
 \end{remark}
  For brevity,  we have used the numbers $0,1,\cdots,r$ by $\gamma\in\Gamma^{*}$ to index each sub-partition-valued function in $\rho$.  More precisely,  we denote $\rho_{\gamma_j}:=\rho^j$.
For $(g,t_{\rho})\in\widetilde{\Gamma}_{\la}$ with $\rho=\rho^0\cup\cdots\cup\rho^r$, one has \cite{HJ}
\begin{equation}
\begin{split}
&\chi_{\la}(g,t_{\rho})=\circledast_{\gamma\in\Gamma^{*}}({\gamma}^{\otimes |\la_{\gamma}|}\otimes \Delta_{\la_{\gamma}})(g,t_{\rho})\\
=&2^{[\frac{|J_{\la}|}{2}]}(\sqrt{-1})^{[\frac{|J_{\la}|}{2}]\cdot\prod_{\gamma\in J_{\la}}d(t_{\rho_{\gamma}})}\cdot\prod_{\gamma\in\Gamma^{*}}
{\gamma}^{\otimes |\la_{\gamma}|}\otimes \Delta_{\la_{\gamma}}(g,t_{\rho_{\gamma}}).
\end{split}
\end{equation}
Let $\tilde{\chi}_{\la}\uparrow$ be the induced character of $\tilde{\chi}_{\la}$ from $\widetilde{\Gamma}_{\la}$ to $\widetilde{\Gamma}_n$.
We have $\tilde{\chi}_{\la}\uparrow=\chi_{\la}\uparrow$ when $d(\la)$ is even and $\tilde{\chi}_{\la}\uparrow=(\chi_{\la}+\chi_{\la}^{'})\uparrow=\chi_{\la}\uparrow+\chi_{\la}^{'}\uparrow$ when $d(\la)$ is odd.
If $\chi_{\la}$ is irreducible, then $\chi_{\la}\uparrow$ is irreducible by Prop. \ref{super-ord}. Hence by Mackey's decomposition theorem and Frobenius reciprocity we obtain (cf. \cite{HJ}):
\begin{equation}\label{eq:innerprod}
\langle \chi_{\la}\uparrow,\chi_{\la}\uparrow\rangle_{\widetilde{\Gamma}_{n}}=\langle\chi_{\la},\chi_{\la}\rangle_{\widetilde{\Gamma}_{\la}}.
\end{equation}

\subsection{Spin character values}
Let $\la=(\la_{\gamma_0},\cdots,\la_{\gamma_r})$ be a strict partition valued function on $\Gamma^{*}$.
Set $\rho=\rho^0\cup\cdots\cup\rho^r$ be the partition-valued function defined in Remark \ref{lem} such that $|\rho^j|=|\la_{\gamma_j}|$ for $j=0,\cdots,r$. From Schur's results in Theorem \ref{Schur}, one sees that if $\Delta_{\la_{\gamma_j}}(t_{\rho^j})$ has a nonzero value then $\rho^j$ must be in $[{\la}_{\gamma_j}]$ for $\gamma_j\in J_{\la}$, and $\rho^j$ must lie in $\mathcal{OSP}_{|\la_{\gamma_j}|}(\Gamma_{*}):=\mathcal{OP}_{|\la_{\gamma_j}|}(\Gamma_{*})\cap \mathcal{SP}_{|\la_{\gamma_j}|}(\Gamma_{*})$ for $\gamma_j\in J_{\la}^{'}$.
then we have the following result.
\begin{proposition} \label{prop:innerprod} Let $\la=(\la_{\gamma})_{\gamma\in\Gamma^{*}}\in \mathcal{SP}_n^1(\Gamma^{*})$. If $\rho=\cup_{j=0}^{r}\rho_{\gamma_j}\in \mathcal{SP}_n^1(\Gamma_{*})$ such that $\rho_{\gamma_j}$ lies in $[{\la}_{\gamma_j}]$ for $\gamma_j\in J_{\la}$, and $\rho_{\gamma_j}$ is in $\mathcal{OSP}_{|\la_{\gamma_j}|}(\Gamma_{*})$ for $\gamma_j\in J_{\la}^{'}$, then
\begin{equation}\label{eq:5}\prod_{\gamma\in J_{\la}^{'}}\big(\sum_{\rho_{\gamma}\in \mathcal{OSP}_{|\rho_{\gamma}|}(\Gamma_{*})}\frac{1}{Z_{\rho_{\gamma}}}|\prod_{c\in\Gamma_{*}}\gamma(c)^{l(\rho_{\gamma}(c))} \Delta_{\la_{\gamma}}(t_{\rho_{\gamma}})|^2\big)=1.
\end{equation}
\end{proposition}
\begin{proof} Write $m=|J_{\la}|$ then we have (Note: here $m$ is odd)
\begin{equation}\label{eq:1}
\begin{split}
&\langle\chi_{\la}\uparrow,\chi_{\la}\uparrow\rangle_{\widetilde{\Gamma}_n-\Gamma\wr \widetilde{A}_n}\\
=&2^{-[\frac{m}{2}]\cdot 2}\sum_{\substack{\rho=\cup_{\gamma\in\Gamma^{*}}\rho_{\gamma}\in \mathcal{SP}^1_n(\Gamma_{*})\\ \rho_{\gamma}\in \mathcal{SP}_{|\la_{\gamma}|}(\Gamma_{*})}}\frac{1}{\widetilde{Z}_{\rho}}|\big(\circledast_{\gamma\in\Gamma^{*}}\gamma^{\otimes |\la_{\gamma}|}\otimes \Delta_{\la_{\gamma}}\big)(D_{\rho})|^2\\
=&2^{-m+1}\cdot\sum_{\substack{\rho=\cup_{\gamma_i\in\Gamma^{*}}\rho_{\gamma_i}\in \mathcal{SP}^1_n(\Gamma_{*})\\ \rho_{\gamma_i}\in \mathcal{SP}_{|\la_{\gamma_i}|}(\Gamma_{*})}}2^{m-2}
\prod_{\gamma\in\Gamma^{*}}\big(\frac{1}{Z_{\rho_{\gamma}}}|\gamma^{\otimes
|\la_{\gamma}|}\otimes \Delta_{\la_{\gamma}}(D_{\rho_{\gamma}})|^2\big)\\
=&\frac{1}{2}\prod_{\gamma\in
\Gamma_{*}}\big(\sum_{\bar{d}(\rho_{\gamma})=\bar{d}(\la_{\gamma}), \rho_{\gamma}\in\mathcal{SP}_{|\la_{\gamma}|}(\Gamma_{*})}
\frac{1}{Z_{\rho_{\gamma}}}|\gamma^{\otimes
|\la_{\gamma}|}\otimes \Delta_{\la_{\gamma}}
(D_{\rho_{\gamma}})|^2\big)\\
\end{split}
\end{equation}
As $Z_{\rho}=\prod_{\gamma\in\Gamma^{*}}Z_{\rho_{\gamma}}$ and $\widetilde{Z}_{\rho_{\gamma}}=2Z_{\rho_{\gamma}}$ for $\gamma\in J_{\la}$,
so Eq. (\ref{eq:1}) becomes

\begin{equation}\label{eq:2}
\begin{split}
=&\frac{1}{2}\prod_{\gamma\in
J_{\la}}\big(\sum_{\rho_{\gamma}\in\mathcal{SP}^1_{|\la_{\gamma}|}(\Gamma_{*})}
\frac{2}{\widetilde{Z}_{\rho_{\gamma}}}|\gamma^{\otimes
|\la_{\gamma}|}\otimes \Delta_{\la_{\gamma}}
(D_{\rho_{\gamma}})|^2\big)\cdot\\
&\prod_{\gamma\in J_{\la}^{'}}\big(\sum_{\rho_{\gamma}\in \mathcal{SP}^0_{|\la_{\gamma}|}
(\Gamma_{*})}\frac{1}{Z_{\rho_{\gamma}}}|\gamma^{\otimes |\la_{\gamma}|}\otimes
\Delta_{\la_{\gamma}}(D_{\rho_{\gamma}})|^2\big)\\
=&2^{m-1}\prod_{\gamma\in J_{\la}}\langle  \gamma^{\otimes |\la_{\gamma}|}\otimes \Delta_{\la_{\gamma}},
\gamma^{\otimes |\la_{\gamma}|}\otimes \Delta_{\la_{\gamma}}\rangle_{\mathcal{SP}^1_{|\la_{\gamma}|}(\Gamma_{*})}\cdot\\
&\prod_{\gamma\in J_{\la}^{'}}\big(\sum_{\rho_{\gamma}\in \mathcal{OSP}_{|\la_{\gamma}|}
(\Gamma_{*})}\frac{1}{Z_{\rho_{\gamma}}}|\gamma^{\otimes |\la_{\gamma}|}\otimes
\Delta_{\la_{\gamma}}(D_{\rho_{\gamma}})|^2\big)\\
=&\frac{1}{2}\prod_{\gamma\in J_{\la}^{'}}\big(\sum_{\rho_{\gamma}\in \mathcal{OSP}_{|\la_{\gamma}|}
(\Gamma_{*})}\frac{1}{Z_{\rho_{\gamma}}}|\gamma^{\otimes |\la_{\gamma}|}\otimes
\Delta_{\la_{\gamma}}(D_{\rho_{\gamma}})|^2\big).
\end{split}
\end{equation}

Thus it follows from \cite{FJW} and Lemma \ref{lema} that
$$\langle  \chi_{\la}\uparrow,\chi_{\la}\uparrow\rangle_{\mathcal{OP}_n(\Gamma_{*})}= \langle  \chi_{\la}\uparrow,\chi_{\la}\uparrow\rangle_{\mathcal{SP}^1_n(\Gamma_{*})}=\frac{1}{2}. $$
Eqs. (\ref{eq:1}) and (\ref{eq:2}) imply that
\begin{equation}
\begin{split}
&\prod_{\gamma\in J_{\la}^{'}}\big(\sum_{\rho_{\gamma}\in \mathcal{OSP}_{|\rho_{\gamma}|}(\Gamma_{*})}\frac{1}{Z_{\rho_{\gamma}}}|\prod_{c\in\Gamma_{*}}\gamma(c)^{l(\rho_{\gamma}(c))} \Delta_{\la_{\gamma}}(z^pt_{\rho_{\gamma}})|^2\big)
=1.
\end{split}
\end{equation}
\end{proof}

The following theorem gives the remaining part of the spin character table for $\widetilde{\Gamma}_n$.
\begin{theorem}
Let $\la=(\la_{\gamma})_{\gamma\in\Gamma^{*}}\in \mathcal{SP}^1_n(\Gamma^{*}) $ and $\rho\in\mathcal{SP}^1_n(\Gamma_{*})$.
(i) If $\rho=\cup_{\gamma\in\Gamma^{*}}\rho_{\gamma}$ such that $\rho_{\gamma}\in [\la_{\gamma}]$
 for $\gamma\in J_{\la}$ and $\rho_{\gamma}\in\mathcal{OSP}_{|\rho_{\gamma}|}(\Gamma_{*})$ for $\gamma\in J_{\la}^{'}$, then

\begin{equation}\nonumber
\begin{split}\chi_{\la}\uparrow(D^{\pm}_{\rho})
=&\pm K_{\rho}\prod_{\gamma\in\Gamma^{*}}
\prod_{c\in\Gamma_{*}}\gamma(c)^{l(\rho_{\gamma}(c))}
\cdot\prod_{\gamma\in J_{\la}^{'}}\Delta_{\la_{\gamma}}(t_{\rho_{\gamma}})\cdot\\
&(\sqrt{-1})^{\Sigma_{\gamma\in J_{\la}}
\frac{|\la_{\gamma}|-l(\la_{\gamma})+1}{2}}\sqrt{\frac{\prod_{\gamma\in
J_{\la}}z_{\la_{\gamma}}}{2}},
\end{split}
\end{equation}
where the value of  $\prod_{\gamma\in
J_{\la}^{'}}\Delta_{\la_{\gamma}}(t_{\rho_{\gamma}})$ is given by
Schur Q-functions (see \cite{FJW}).

(ii) $\chi_{\la}\uparrow(D_{\rho}^{\pm})=0$, otherwise.
\end{theorem}
\begin{proof}(i)  Let $T$ be a left coset of $\widetilde{\Gamma}_{\la}$ in $\widetilde{\Gamma}_n$
and $K_{\rho}$ is the number of left cosets $T$ of $\widetilde{\Gamma}_{\la}$ in $\widetilde{\Gamma}_n$
such that $(g, t_{\rho})T=T$. Then $t^{-1}(g, t_{\rho})t\in\widetilde{\Gamma}_{\lambda}$
for any left coset representative $t\in T$. It follows from the formula of induced character that
\begin{equation}\label{eq:3}
 \begin{split}
\chi_{\la}\uparrow(D^{\pm}_{\rho})=&\pm\frac{1}{|\widetilde{\Gamma}_{\la}|}
\sum_{T\in\widetilde{\Gamma}_n/\widetilde{\Gamma}_{\lambda}}\big(\sum_{t\in T}\circledast_{\gamma\in\Gamma^{*}}
(\gamma^{\otimes|\la_{\gamma}|}\otimes\Delta_{\la_{\gamma}})(t^{-1}(g, t_{\rho})t)\big)\\
=&\pm K_{\rho}\cdot\circledast_{\gamma\in\Gamma^{*}}
(\gamma^{\otimes|\la_{\gamma}|}\otimes\Delta_{\la_{\gamma}})(g, t_{\rho})\\
=&\pm K_{\rho} \cdot 2^{\frac{|J_{\la}|-1}{2}}
\prod_{\gamma\in\Gamma^{*}}
(\gamma^{\otimes|\la_{\gamma}|}\otimes\Delta_{\la_{\gamma}})(D^{+}_{\rho_{\gamma}}).\\
\end{split}
\end{equation}
In the above we have used
$ \chi_{\la}\uparrow(D^{-}_{\rho})=- \chi_{\la}\uparrow(D^{+}_{\rho})$, and the second line
vanishes if $t_{\rho}\notin\widetilde{S}_{\lambda}$.  Moreover, by \cite{FJW},
$$ \gamma^{\otimes|\la_{\gamma}|}\otimes\Delta_{\la_{\gamma}}(D^{+}_{\rho_{\gamma}})
=\prod_{c\in\Gamma_{*}}\gamma(c)^{l(\rho_{\gamma}(c))}\cdot\Delta_{\la_{\gamma}}(t_{\rho_{\gamma}}).$$
So  Eq.  (\ref{eq:3}) is equal to
\begin{equation}
\begin{split}
=&\pm K_{\rho}\cdot 2^{\frac{|J_{\la}|-1}{2}}
\prod_{\gamma\in\Gamma^{*}}
 \prod_{c\in\Gamma_{*}}\gamma(c)^{l(\rho_{\gamma}(c))}\cdot\prod_{\gamma\in \Gamma^{*}}\Delta_{\la_{\gamma}}(t_{\rho_{\gamma}})\\
=&\pm K_{\rho}\cdot2^{\frac{|J_{\la}|-1}{2}}
\prod_{\gamma\in\Gamma^{*}}
\prod_{c\in\Gamma_{*}}\gamma(c)^{l(\rho_{\gamma}(c))}\cdot\prod_{\gamma\in
J_{\la}^{'}} \Delta_{\la_{\gamma}}(t_{\rho_{\gamma}})\cdot\\
& \qquad\qquad\prod_{\gamma\in
J_{\la}}(\sqrt{-1})^{\frac{|\rho_{\gamma}|-l(\rho_{\gamma})+1}{2}}
\sqrt{\frac{\prod_{c\in\Gamma_{*}}z_{\rho_{\gamma}(c)}}{2}}\\
=&\pm K_{\rho}\prod_{\gamma\in\Gamma^{*}}\prod_{c\in\Gamma_{*}}\gamma(c)^{l(\rho_{\gamma}(c))}
\cdot\prod_{\gamma\in J_{\la}^{'}}\Delta_{\la_{\gamma}}(t_{\rho_{\gamma}})\cdot\\
& \qquad\qquad(\sqrt{-1})^{\Sigma_{\gamma\in J_{\la}}
\frac{|\la_{\gamma}|-l(\la_{\gamma})+1}{2}}\sqrt{\frac{\prod_{\gamma\in
J_{\la}}z_{\la_{\gamma}}}{2}}.
\end{split}
\end{equation}

(ii) Note that we can assume that $\rho=\cup_{\gamma\in\Gamma^{*}}\rho_{\gamma}$, where $\rho_{\gamma}\in \mathcal{SP}_{|\la_{\gamma}|}(\Gamma_{*})$.
Otherwise it is easy to see that $\chi_{\la}\uparrow(\rho)=0$.

(1) If $\rho_{\gamma}\in\mathcal{OSP}_{|\rho_{\gamma}|}(\Gamma_{*})$ for $\gamma\in J_{\la}^{'}$, we show that $\rho_{\gamma}\in [\la_{\gamma}]$ for $\gamma\in J_{\la}$. We denote $m=|J_{\la}|$. It follows from Proposition \ref{prop:innerprod} that
\begin{equation}\label{eq:4}
\begin{split}
&\langle\chi_{\la}\uparrow,\chi_{\la}\uparrow\rangle_{\mathcal{SP}^1_n(\Gamma_{*})}\\
=&2^{m-1}\prod_{\gamma\in
J_{\la}}\big(\sum_{\rho_{\gamma}\in\mathcal{SP}^1_{|\la_{\gamma}|}(\Gamma_{*})}
 \frac{1}{\widetilde{Z}_{\rho_{\gamma}}}|\prod_{c\in\Gamma_{*}}\gamma(c)^{l(\rho_{\gamma}(c))}
  \Delta_{\la_{\gamma}}(z^pt_{\rho_{\gamma}})|^2\big)\\
\geq&2^{m-1}\prod_{\gamma\in J_{\la}}\big(\sum_{\rho_{\gamma}\in
[\la_{\gamma}]}
\frac{1}{Z_{\rho_{\gamma}}}|\prod_{c\in\Gamma_{*}}\gamma(c)^{l(\rho_{\gamma}(c))}
 \Delta_{\la_{\gamma}}(t_{\rho_{\gamma}})|^2\big)\\
=&2^{m-1}\prod_{\gamma\in J_{\la}}\big(\sum_{\rho_{\gamma}\in [\la_{\gamma}]} \big(\prod_{c\in\Gamma_{*}}
\frac{|\gamma(c)|^{2l(\rho_{\gamma}(c))}}
{z_{\rho_{\gamma}(c)}\zeta_c^{l(\rho_{\gamma}(c))}}\big)\left|\sqrt{\frac{\prod_{c\in\Gamma_{*}}z_{\rho_{\gamma}(c)}}{2}}\right|^2\big)\\
=&2^{m-1}\prod_{\gamma\in J_{\la}}\big(\sum_{\rho_{\gamma}\in [\la_{\gamma}]}\frac{1}{2}
\frac{\prod_{c\in\Gamma_{*}}|\gamma(c)|^{2l(\rho_{\gamma}(c))}}
{\prod_{c\in\Gamma_{*}}\zeta_c^{l(\rho_{\gamma}(c))}}\big)\\
\end{split}
\end{equation}

  Since  $\rho_{\gamma}\in [\la_{\gamma}]$, $l(\rho_{\gamma})=l(\la_{\gamma})$  and $$\gamma(c)^{l(\rho_{\gamma}(c))}=\gamma^{\otimes l(\rho_{\gamma}(c))}\cdot
  {(\underset{l(\rho_{\gamma}(c))}{\underbrace{c,\cdots,c}})},$$
hence we have that
\begin{equation}
\begin{split}
&\prod_{c\in\Gamma_{*}}\gamma(c)^{l(\rho_{\gamma}(c))}
=\gamma^{\otimes l(\rho_{\gamma})}\cdot
(\underset{l(\rho_{\gamma}(c^0))}{\underbrace{c^0,\cdots,c^0}},\cdots,\underset{l(\rho_{\gamma}(c^r))}
{\underbrace{c^r,\cdots,c^r}})\\
=&\gamma^{\otimes
l(\rho_{\gamma}(c^0))}\otimes\cdots\otimes\gamma^{\otimes
l(\rho_{\gamma}(c^r))}\cdot
(\underset{l(\rho_{\gamma}(c^0))}{\underbrace{c^0,\cdots,c^0}},\cdots,\underset{l(\rho_{\gamma}(c^r))}
{\underbrace{c^r,\cdots,c^r}})\\
=&\gamma^{\otimes l(\rho_{\gamma})}\cdot
(c^{i_1},c^{i_2},\cdots,c^{i_{l(\rho_{\gamma})}})\\
\end{split}
\end{equation}
which satisfies that the number of $c^{j} (j=0,\cdots,r)$ in $\{c^{i_1},c^{i_2},\cdots,c^{i_{l(\rho_{\gamma})}}\}$ is equal to $l(\rho_{\gamma}(c^j))$.
Clearly $C_{\rho_{\gamma}}=(c^{i_1},c^{i_2},\cdots,c^{i_{l(\rho_{\gamma})}})$ is a
conjugacy class of  $\Gamma^{l(\la_{\gamma})}$. Let
$\zeta_{C_{\rho_{\gamma}}}$ be the order of  the centralizer of an
element in the conjugacy class $C_{\rho_{\gamma}}$ of
$\Gamma^{l(\la_{\gamma})}$, then
$\zeta_{C_{\rho_{\gamma}}}=\prod_{c\in\Gamma_{*}}\zeta_c^{l(\rho_{\gamma}(c))}$.
As $\rho_{\gamma}\in [\la_{\gamma}]$, so $C_{\rho_{\gamma}}$ can run through all conjugacy classes in $\Gamma^{l(\la_{\gamma})}$. For example,
if we assume $[\la_{\gamma}]=\{\rho_{\gamma}\in \mathcal{SP}_{11}^1(\Gamma_{*})|\la_{\gamma}=(5, 4,2)\}$, then $\rho_{\gamma}=(5_{\square_1}, 4_{\square_2}, 2_{\square_3})$ and each $\square_j$ can run from $c^0$ to $c^r$. So $C_{\rho_{\gamma}}=(\square_1, \square_2, \square_3)$ runs through
all conjugacy classes of $\Gamma^{3}$.
Subsequently Eq. (\ref{eq:4}) becomes
\begin{equation}
\begin{split}
=&2^{m-1}\prod_{\gamma\in J_{\la}}(\sum_{C_{\rho_{\gamma}}\in (\Gamma^{l(\la_{\gamma})})_{*}}
\frac{1}{2\zeta_{C_{\rho_{\gamma}}}}\cdot\\
& |\gamma^{\otimes l(\rho_{\gamma}(c^0))}\otimes\cdots\otimes\gamma^{\otimes l(\rho_{\gamma}(c^r))}\cdot
(\underset{l(\rho_{\gamma}(c^0))}{\underbrace{c^0,\cdots,c^0}},\cdots,\underset{l(\rho_{\gamma}(c^r))}
{\underbrace{c^r,\cdots,c^r}})|^2)\\
=&2^{m-1}\prod_{\gamma\in J_{\la}}\big(\sum_{C_{\rho_{\gamma}}\in (\Gamma^{l(\la_{\gamma})})_{*}}
\frac{1}{2\zeta_{C_{\rho_{\gamma}}}}\cdot|\gamma^{\otimes l(\rho_{\gamma})}(c^{i_1},c^{i_2},\cdots,c^{i_{l(\rho_{\gamma})}})|^2\big)\\
=&2^{m-1}\prod_{\gamma\in J_{\la}}(\frac{1}{2}\sum_{C_{\rho_{\gamma}}\in (\Gamma^{l(\la_{\gamma})})_{*}}
\frac{1}{\zeta_{C_{\rho_{\gamma}}}}|\gamma^{\otimes l(\la_{\gamma})}(C_{\rho_{\gamma}})|^2)\\
=&2^{m-1}\prod_{\gamma\in J_{\la}}\frac{1}{2}\langle  \gamma^{\otimes l(\la_{\gamma})},
\gamma^{\otimes l(\la_{\gamma})}\rangle_{\Gamma^{l(\la_{\gamma})}}
=\frac{1}{2},
\end{split}
\end{equation}
which forces $\chi_{\la}\uparrow(D^{\pm}_{\rho})=0$ if  $\rho_{\gamma}\notin [\la_{\gamma}]$ for $\gamma\in J_{\la}$.

(2) If $\rho_{\gamma}\notin\mathcal{OSP}_{|\la_{\gamma}|}(\Gamma_{*})$ for $\gamma\in J_{\la}^{'}$, then there is at least one $\rho_{\gamma}$ not in $\mathcal{OP}_{|\la_{\gamma}|}(\Gamma_{*})$ for $\gamma\in J_{\la}^{'}$. Meanwhile, $\Delta_{\la_{\gamma}}$ is a double spin character
 when $\gamma\in J_{\la}^{'}$, so we have
$\Delta_{\la_{\gamma}}(t_{\rho_{\gamma}})=0$, thus $\chi_{\la}\uparrow(D^{\pm}_{\rho})=0$.
This completes the proof.
\end{proof}

\bigskip

\centerline{\bf Acknowledgments}
The second author
gratefully acknowledges the partial support of Max-Planck Institut f\"ur Mathematik in Bonn, Simons Foundation,
NSFC 
and NSF during this work.

\bibliographystyle{amsalpha}

\end{document}